\documentclass[reqno,10pt]{amsart}

\usepackage{amsthm,amsmath,amssymb,amscd}
\usepackage{graphics} 
\usepackage[all,cmtip]{xy}
\usepackage{graphicx}
\usepackage{subfigure}
\usepackage{url}
\usepackage{hyperref}
\usepackage{color}
\usepackage[usenames,dvipsnames]{xcolor}
\usepackage{verbatim}
\usepackage{fancyhdr}
\usepackage{geometry}
\usepackage{cancel}
\usepackage{mathtools}
\usepackage[english]{babel}

\usepackage{setspace}
\def\vol{\mathop{\hbox{vol}}}
\def\det{\mathop{\hbox{det}}}

\newcommand{\real}{\Bbb{R}}

\newcommand{\complex}{\Bbb{C}}
\newcommand{\C}{\Bbb{C}}
\newcommand{\projective}{\Bbb{P}}

\newcommand{\intero}{\Bbb{Z}}

\newcommand{\K}{K\"{a}hler}

\newtheorem{teor}{Theorem}[section]
\newtheorem{defin}[teor]{Definition}
\newtheorem{remar}[teor]{Remark}
\newtheorem{prop}[teor]{Proposition}
\newtheorem{corol}[teor]{Corollary}
\newtheorem{lemma}[teor]{Lemma}

\newtheorem{Example}[teor]{Example}

\begin{document}

\title[Symplectic and K\"{a}hler geometry of toric varieties]{Some remarks on the symplectic and K\"{a}hler geometry of toric varieties}

\author{Claudio Arezzo}
\address{Abdus Salam International Center for Theoretical Physics \\
                  Strada Costiera 11 \\
         Trieste (Italy) and Dipartimento di Matematica \\
         Universit\`a di Parma \\
         Parco Area delle Scienze~53/A  \\
         Parma (Italy)}
\email{arezzo@ictp.it}

\author{Andrea Loi}
\address{(Andrea Loi) Dipartimento di Matematica \\
         Universit\`a di Cagliari (Italy)}
         \email{loi@unica.it}

\author{Fabio Zuddas}
\address{(Fabio Zuddas) Dipartimento di Matematica e Informatica \\
          Via delle Scienze 206 \\
         Udine (Italy)}
\email{fabio.zuddas@uniud.it}

\thanks{
The second author was  partially supported by Prin 2010/11 -- Variet\`a reali e complesse: geometria, topologia e analisi armonica -- Italy; the first and third authors were partially supported by the FIRB Project ``Geometria Differenziale Complessa e Dinamica Olomorfa''.}

\begin{abstract}
Let $M$ be a projective toric manifold. We prove two results concerning respectively \K-Einstein submanifolds of $M$ and symplectic embeddings of the standard euclidean ball in $M$. Both results use the well-known fact that $M$ contains an open dense subset biholomorphic to $\complex^n$.
\end{abstract}
 
\maketitle

\section{Introduction and statements of the main results}\label{Introduction}

In this paper we use the well-known fact that toric manifolds are compactifications of $\complex^n$ in order to prove two results, of Riemannian and symplectic nature, given by the following two theorems.

\begin{teor} 
Let $N$ be a projective toric manifold equipped with a toric \K \ metric $G$ and $(M, g) \xhookrightarrow{\phi} (N, G)$ be an isometric embedding of a \K -Einstein manifold such that $\phi(M)$ contains a point of $N$ fixed by the torus action. Then $(M, g)$ has positive scalar curvature.
\end{teor}

\begin{teor}
Let $(M, \omega)$ a toric manifold endowed with an integral toric \K \ form and let $\Delta \subseteq \real^n$ be the image of the moment map for the torus action. Then, there exists a number $c(\Delta)$ (explicitely computable from the polytope, see Corollary \ref{corollariodivisore}) such that any ball of radius $r > c(\Delta)$, symplectically embedded into $(M, \omega)$, must intersect the divisor $M \setminus \complex^n$.
\end{teor}

These two results are proved and discussed respectively in Section 2 (Theorem \ref{theotoriche}) and Section 3 (Corollary \ref{corollariodivisore}).

The paper ends with an Appendix where, for the reader's convenience, we give an exposition (as self-contained as possible) of the classical facts about toric manifolds  we need in Sections 2 and 3.

\section{\K--Einstein submanifolds of Toric manifolds}

Let us briefly recall Calabi's work on \K \ immersions and diastasis function \cite{ca}.

Given a complex manifold $N$ endowed with a real analytic \K \ metric $G$, the ingenious idea of 
Calabi was the introduction,
in a neighborhood of a point 
$p\in N$,
of a very special
\K\ potential
$D_p$ for the metric
$G$, which 
he christened 
{\em diastasis}.
Recall that
a \K\ potential
is an
analytic function 
$\Phi$ 
defined in a neighborhoood
of a point $p$
such that 
$\Omega =\frac{i}{2}\partial \bar\partial\Phi$,
where $\Omega$
is the \K\ form
associated
to $G$.
In a complex coordinate system
$(Z)$ around $p$
$$G_{\alpha\beta}=
2G(\frac{\partial}{\partial Z_{\alpha}},
\frac{\partial}{\partial \bar Z_{\beta}})
=\frac{{\partial}^2\Phi}
{\partial Z_{\alpha}\partial\bar Z_{\beta}}.$$
 
A \K\ potential is not unique:
it is defined up to the sum with
the real part of a holomorphic function.
By duplicating the variables $Z$ and $\bar Z$
a potential $\Phi$ can be complex analytically
continued to a function 
$\tilde\Phi$ defined in a neighborhood
$U$ of the diagonal containing
$(p, \bar p)\in N\times\bar N$
(here $\bar N$ denotes the manifold
conjugated of $N$).
The {\em diastasis function} is the 
\K\ potential $D_p$
around $p$ defined by
$$D_p(q)=\tilde\Phi (q, \bar q)+
\tilde\Phi (p, \bar p)-\tilde\Phi (p, \bar q)-
\tilde\Phi (q, \bar p).$$
Among all the potentials the diastasis
is characterized by the fact that 
in every coordinates system 
$(Z)$ centered in $p$
$$D_p(Z, \bar Z)=\sum _{|j|, |k|\geq 0}
a_{jk}Z^j\bar Z^k,$$
with 
$a_{j 0}=a_{0 j}=0$
for all multi-indices
$j$.
The following proposition shows
the importance of
the diastasis
in the context of holomorphic
maps between 
\K\ manifolds.
\begin{prop}{\bf (Calabi)}\label{calabidiastasis}
Let 
$\varphi :(M, g)\rightarrow
(N, G)$
be a holomorphic and isometric
embedding between \K\ manifolds
and suppose that $G$ is real analytic.
Then $g$ is real analytic
and for every point 
$p\in M$
$$\varphi (D_p)=D_{\varphi (p)},$$
where $D_p$ (resp. $D_{\varphi (p)})$
is the diastasis of $g$ relative to
$p$ (resp. of $G$ relative to $\varphi (p))$. 
\end{prop}

In Proposition \ref{teormain} below, we are going to require that $N$ is a compactification of $\complex^n$, or more precisely that $N$ contains an analytic subvariety  $Y$ such that $X = N \setminus Y$ is biholomorphic to $\complex^n$; as far as the \K \ metric $G$ on $N$ is concerned, in addition to the requirement that $G$ is real analytic, we impose two other conditions. The first one is

\smallskip

{\em Condition $(A)$: there exists
a point $p_*\in X=N\setminus Y$ such that
the diastasis $D_{{p_*}}$ is globally
defined and non-negative on $X$.} 

\vskip 0.3cm

In order to describe the second condition
we need to introduce the concept of
Bochner's coordinates (cfr. \cite{boc}, \cite{ca}, \cite{hu},
\cite{hu1}).
Given a real analytic
\K\ metric $G$ on $N$
and a point $p\in N$,
one can always find local 
(complex) coordinates
in a neighborhood of 
$p$
such that
$$D_p(Z, \bar Z)=|Z|^2+
\sum _{|j|, |k|\geq 2}
b_{jk}Z^j\bar Z^k,$$
where 
$D_p$ is the diastasis
relative to $p$.
These coordinates,
uniquely defined up
to a unitary transformation,
are called 
{\em the Bochner's coordinates}
with respect to the point $p$.

One important feature of these
coordinates which we are going to use
in the proof of our main theorem
is the following:
\begin{teor}({\bf Calabi})\label{calabibochner}
Let 
$\varphi :(M, g)\rightarrow
(N, G)$
be a holomorphic and isometric
embedding between \K\ manifolds
and suppose that $G$ is real analytic.
If $(z_1,\dots ,z_m)$
is a system of Bochner's coordinates in
a neighborhood $U$
of $p\in M$ then there exists
a system of 
Bochner's coordinates
$(Z_1,\dots ,Z_n)$
with respect to $\varphi (p)$
such that
\begin{equation}\label{zandZ}
Z_1|_{\varphi(U)}=z_1,\dots ,Z_m|_{\varphi (U)}=z_m.
\end{equation}
\end{teor}

\smallskip

We can then state the following

\smallskip

\noindent
{\em Condition $(B)$: the Bochner's 
coordinates with respect to the point
$p_*\in X$, given by the previous condition $(A)$,
are globally defined on 
$X$.}

\vskip 0.3cm

\smallskip

Our first result is then the following:

\begin{prop}\label{teormain}
Let 
$N$ be 
a smooth
projective 
compactification 
of $X$ 
such that
$X$
is algebraically
biholomorphic to
${\complex}^n$
and let $G$
be a real analytic \K\
metric on $N$
such that the following two 
conditions are satisfied:
\begin{itemize}
\item [(A)]
there exists
a point $p_*\in X$ such that
the diastasis $D_{{p_*}}$ is globally
defined and non-negative on $X$;
\item [(B)]
the Bochner's 
coordinates with respect to 
$p_*$
are globally defined on 
$X$.
\end{itemize}
Then any
K--E submanifold
$(M, g) \xhookrightarrow{\phi} (N, G)$ such that $p_* \in \phi(M)$ has positive
scalar curvature.
\end{prop}

\begin{remar}\label{condA}
\rm The easiest example of compactification of
${\complex}^n$ which satisfies condition $(A)$
is given by 
${\complex}P^n={\complex}^n\cup
Y$ endowed with the Fubini--Study metric
$g_{FS}$, namely the metric whose associated \K\ form is given by
\begin{equation}
\omega_{FS}=\frac{i}{2}
\partial\bar{\partial}\log\displaystyle
\sum _{j=0}^{n}|Z_{j}|^{2},
\end{equation}
and
$Y={\complex}P^{n-1}$
is the hyperplane 
$Z_0=0$.
Indeed the diastasis
with respect to
$p_*=[1, 0,\dots ,0]$
is given by:
$$D_{p_*}(u, \bar{u})=
\log (1+\sum _{j=1}^{n}|u_j|^2).$$
where $(u_1,\dots ,u_n)$ are the affine coordinates, 
namely 
$u_j=\frac{Z_j}{Z_0}, j=1,\dots, n$.
Proposition \ref{teormain} can be then  considered as an extension of
a theorem of Hulin \cite{hu1} which asserts that a compact  \K--Einstein submanifold
of ${\complex}P^n$
is Fano (see also \cite{LZ09}).

Other  examples  of compactifications of
${\complex}^n$
satisfying conditions
$(A)$ and $(B)$ are given by the compact homogeneous Hodge
manifolds. 
These are not interesting since all compact homogeneous 
Hodge manifolds can be \K\ embedded into  a complex projective space (\cite{DLH})
and so we are reduced to study the Hulin's problem.
We also remark that, by Proposition \ref{calabidiastasis},
condition $(A)$ is satisfied also by all
the \K\ submanifolds of the previous examples.
\end{remar}

\begin{proof}[Proof of Proposition \ref{teormain}]
Let 
$p$
be a point
in $M$ 
such that 
$\varphi (p)=p_*$,
where $p_*$ is the point in 
$N$ given by condition $(A)$.
Take Bochner's 
coordinates 
$(z_1,\dots ,z_m)$
in a neighborhood 
$U$ of 
$p$ which we 
take small enough to
be contractible.
Since the \K\ metric
$g$ is Einstein with
(constant) scalar
curvature $s$ then:
$\rho_{\omega}=\lambda \omega$
where 
$\lambda$ is the Einstein
constant, i.e. 
$\lambda=\frac{s}{2m}$,
and $\rho_{\omega}$
is the \emph{Ricci form}.
If $\omega =\frac{i}{2}
\sum _{j=1}^{m}g_{j\bar{k}}
dz_{j}\wedge d\bar{z}_{\bar{k}}$
then 
$\label{rholocal}
\rho _{\omega}=-i\partial
\bar{\partial}\log \det g_{j\bar{k}}$
is the local expression of
its \emph{Ricci form}.

Thus the volume form
of $(M, g)$ reads
on 
$U$ as:
\begin{equation}\label{voleucl}
\frac{\omega^m}{m!}=\frac{i^m}{2^m}
e^{-\frac{\lambda}{2}D_p+F+\bar F}
dz_1\wedge d\bar z_1\wedge\dots\wedge 
dz_m\wedge d\bar z_m\, \, ,
\end{equation}
where $F$ is a holomorphic
function on $U$
and
$D_p=\varphi^{-1}(D_{p_*})$
is the diastasis on $p$
(cfr. Proposition
\ref{calabidiastasis}).
 
We claim that
$F+\bar F=0$.
Indeed, 
observe that 
$$\frac{\omega ^m}{m!}=\frac{i^m}{2^m}
\det (\frac{\partial ^2 D_p}
{\partial z_{\alpha} \partial \bar z_{\beta}})
dz_1\wedge d\bar z_1\wedge\dots\wedge 
dz_m\wedge d\bar z_m .$$
By the very definition
of Bochner's coordinates it is easy
to check that 
the expansion of 
$\log\det (\frac{\partial ^2 D_p}
{\partial z_{\alpha} \partial \bar z_{\beta}})$
in the 
$(z,\bar z)$-coordinates
contains only mixed terms
(i.e. of the form
$z^j\bar z^k, j\neq 0, k\neq 0$).
On the other hand
by formula  (\ref{voleucl})
$$-\frac{\lambda}{2} D_p + F + \bar F=
\log\det (\frac{\partial ^2 D_p}
{\partial z_{\alpha} \partial \bar z_{\beta}}).$$

Again by the definition of the Bochner's 
coordinates this forces
$F + \bar F$ to be zero,
proving our claim.
By Theorem \ref{calabibochner}
there exist Bochner's
coordinates 
$(Z_1, \dots , Z_n)$
in a neighborhood of 
$p_*$
satisfying 
(\ref{zandZ}).
Moreover, by condition $(B)$
this coordinates 
are 
globally defined
on
$X$.
Hence, by 
formula (\ref{voleucl})
(with $F+\bar F=0$),
the
$m$-forms 
$\frac{\Omega^m}{m!}$
and
$e^{-\frac{\lambda}{2}D_{p_*}}
dZ_1\wedge d\bar Z_1\wedge\dots\wedge 
dZ_m\wedge d\bar Z_m$ 
globally defined on $X$
agree on the open set
$\varphi (U)$.
Since they are
real analytic  they
must agree
on the
connected open set
$\hat M=\varphi (M)\cap X$, i.e.
\begin{equation}\label{eqforms}
\frac{\Omega^m}{m!}=\frac{i^m}{2^m}
e^{-\frac{\lambda}{2}
D_{p_*}}
dZ_1\wedge d\bar Z_1\wedge\dots\wedge 
dZ_m\wedge d\bar Z_m.
\end{equation}

Since 
$\frac{\Omega^m}{m!}$
is a volume form on $\hat M$
we deduce that
the restriction of the
projection map
$$\pi :X\cong{\complex}^n
\rightarrow {\complex}^m:
(Z_1,\dots ,Z_n)\mapsto 
(Z_1,\dots ,Z_m)$$
to $\hat M$ 
is open.
Since it
is also algebraic 
(because the biholomorphism
between $X$ and  ${\complex}^n$
is algebraic),
its image 
contains a Zariski open
subset of ${\complex}^m$
(see Theorem
$13.2$ in \cite{bo}),
hence 
its euclidean volume,
$\vol_{eucl}(\pi (\hat M))$,  
has to be infinite.
Suppose now
that the scalar curvature
of $g$ is non-positive.
By formula 
(\ref{eqforms})
and  by the fact that $D_{p_*}$ is non-negative,
we get 
$\vol (\hat M, g)\geq
\vol_{eucl}(\pi (\hat M))$
which is the desired contradiction,
being the volume of $M$
(and hence that of 
$\hat M$)
finite.
\end{proof}

\vskip 0.3cm

Now, we are going to apply Proposition \ref{teormain} to {\it toric manifolds} endowed with {\it toric K\"ahler metrics}.

Recall that a toric manifold $M$ is a complex manifold which contains an open dense subset biholomorphic to $(\complex^*)^n$ and such that the canonical action of $(\complex^*)^n$ on itself by $(\alpha_1, \dots, \alpha_n)(\beta_1, \dots, \beta_n) = (\alpha_1 \beta_1, \dots, \alpha_n \beta_n)$ extends to a holomorphic action on the whole $M$ (see the Appendix for more details).
A toric K\"ahler metric $\omega$ on $M$ is a K\"ahler metric which is invariant for the action of the real torus $T^n = \{ (e^{i \theta_1}, \dots , e^{i \theta_n}) \ | \ \theta_i \in \real \}$ contained in the dense, complex torus $({\complex}^*)^n$ , that is for every fixed $\theta \in T^n$ the diffeomorphism $f_{\theta}: M \rightarrow M$ given by the action of $(e^{i \theta_1}, \dots , e^{i \theta_n})$ is an isometry.

We have the following, well-known fact (compare, for example, with Section 2.2.1 in \cite{donaldson} or Proposition 2.18 in \cite{batyrev}).

\begin{prop}\label{apertointoriche}
If $M$ is a projective, compact toric manifold then there exists an open dense subset $X \subseteq M$ which is algebraically biholomorphic to $\complex^n$. More precisely, for every point $p \in M$ fixed by the torus action there are an open dense neighbourhood $X_p$ of $p$ and a biholomorphism $\phi_p: X_p \rightarrow \complex^n$ such that $p$ is sent to the origin and the restriction of the torus action to $X_p$ corresponds via $\phi$ to the canonical action of $(\complex^*)^n$ on $\complex^n$. 
\end{prop}

A self-contained proof of this proposition in given in Section \ref{supapp1} of the Appendix. Now we are ready to prove the main result of this section.

\begin{teor}\label{theotoriche}
Let $N$ be a projective toric manifold equipped with a toric \K \ metric $G$. Then any K-E submanifold $(M, g) \xhookrightarrow{\phi} (N, G)$ such that $\phi(M)$ contains a point of $N$ fixed by the torus action has positive scalar curvature.
\end{teor}

\begin{proof}
As we have just recalled, $N$ is a smooth projective compactification of an open subset algebraically biholomorphic to $\complex^n$. So, the Theorem will follow from Proposition \ref{teormain} once we have shown that, for $p_*$ equal to a point $N$ fixed by the torus action, then the conditions (A) and (B) of the statement of Proposition \ref{teormain} are satisfied. 

Let then $p_* \in N$ be such a point, and let $\xi = (\xi_1, \dots, \xi_n)$ be the system of coordinates given by the biholomorphism $\phi_{p_*}: X_{p_*} \rightarrow \complex^n$ given in Proposition \ref{apertointoriche} above. 

Let $\Omega$ be the \K \ form associated to $G$ and let $\Phi$ be a local potential for $\Omega$ around the origin  in the coordinates $\xi = (\xi_1, \dots, \xi_n)$.  Since $X = X_{p_*}$ is contractible, $\Phi$ can be extended to all $X$ (see, for example, Remark 2.6.2 in \cite{GMS}) and
$$D(\xi, \bar \xi) = \Phi(\xi, \bar \xi) + \Phi(0,0) - \Phi(0, \bar \xi) - \Phi(\xi, 0)$$
is a diastasis function on all $X$ in the coordinates $\xi_1, \dots, \xi_n$.

For every $\theta \in T^n$ and $\xi \in \complex^n$,  let us denote 
$$e^{i \theta} \xi := (e^{i \theta_1}, \dots , e^{i \theta_n})(\xi_1, \dots, \xi_n) =  (e^{i \theta_1} \xi_1, \dots , e^{i \theta_n} \xi_n)$$ and $D_{\theta}(\xi, \bar \xi) := D(e^{i \theta} \xi, e^{-i \theta} \bar \xi).$
Then
$$i \partial \bar \partial D_{\theta} =  i \frac{\partial^2 D_{\theta}}{\partial \xi_k \partial \bar \xi_l}(\xi, \bar \xi) d\xi_k \wedge d \bar \xi_l = i e^{i(\theta_k - \theta_l)}\frac{\partial^2 D}{\partial \xi_k \partial \bar \xi_l}(e^{i \theta} \xi, e^{-i \theta} \bar \xi) d\xi_k \wedge d \bar \xi_l =$$

$$ = i \frac{\partial^2 D}{\partial \xi_k \partial \bar \xi_l}(e^{i \theta} \xi, \overline{e^{i \theta}\xi}) d(e^{i\theta_k} \xi_k) \wedge d \overline{(e^{i\theta_l} \xi_l)} = (e^{i \theta})^*((\xi)^*(\Omega|_X))=$$
$$ = (\xi)^*(f_{\theta}^*(\Omega|_X)) = (\xi)^*(\Omega|_X),$$

where the last equality follows by the invariance of $\Omega$ for the action of $T^n$. Then, for every $\theta \in T^n$, the function $D_{\theta}$ is a potential for $\Omega$ on $X$; moreover, it clearly satisfies the characterization for the diastasis. By the uniqueness of the diastasis around the origin, we then have $D = D_{\theta}$, that is 
$$D(\xi, \bar \xi) = D(e^{i \theta_1} \xi_1, \dots, e^{i \theta_n} \xi_n, e^{-i \theta_1} \bar \xi_1, \dots, e^{-i \theta_n} \bar \xi_n) .$$ 
This last equality means that $D$ depends on the norms $|\xi_1|^2, \dots, |\xi_n|^2$ (i.e. $D$ is a {\it rotation invariant} function), and in particular it is immediately seen to satisfy the condition for $\xi_1, \dots, \xi_n$ to be Bochner coordinates.

In order to show that $D$ is non-negative, recall that, since $i \partial \bar \partial D$ is a K\"ahler form, $D$ must be a plurisubharmonic function, which means that, for any $a = (a_1, \dots, a_n), b = (b_1, \dots, b_n) \in \C^n$, the function of one complex variable $f(\xi) = D(a\xi+b) = D(a_1\xi+b_1, \dots, a_n \xi + b_n)$ is a subharmonic function, i.e. $\frac{\partial^2 f}{\partial \xi \partial \bar \xi} \geq 0$. To prove the claim it will be enough to show that, for any $a \in \C^n$, the rotation invariant subharmonic function $f_a(\xi) = D(a\xi)$ is non-negative. Now, we have

$$0 \leq \frac{\partial^2 f_a}{\partial \xi \partial \bar \xi} = t \cdot \frac{d^2 f_a}{d t^2}  + \frac{d f_a}{d t} = \frac{d}{d t}(t f_a(t))$$
where we are denoting $t = |\xi|^2$.

It follows that $g(t) = t f_a(t)$ is a non-decreasing function, and since $g(0)=0$ we have $g(t) = t f_a(t) \geq 0$, that is $f_a(t) \geq 0$, as required.

\end{proof}

\begin{remar}
\rm If $\phi(M)$ does not contain any point of $N$ fixed by the torus action, then for any $f \in Aut(N) \cap Isom(N,G)$ one could be tempted to replace $\phi$ by $f \circ \phi$ (which is clearly again a \K\ embedding) so to have that $f(\phi(M))$ contains a fixed point.

Anyway, while the automorphisms group of a toric manifold can be explicitly described, in general we do not have control on $Isom(N, G)$, and in general this group can be too small. For example, for the \K-Einstein metric $G$ on $N = \complex \projective^2 \sharp 3 \overline{\complex \projective^2}$ (\cite{Siu}, \cite{Tian}), one has that $Isom(N,G)$ is the real part of $Aut(N)$, whose component of the identity $Aut^{\circ}(N)$ contains only the automorphisms given by the action of the complex torus $(\complex^*)^n$ (indeed, one easily sees that the set of the {\it Demazure roots} is empty in this case, see for example Section 3.4 in \cite{Oda})), so $Isom(N, G) \simeq T^n$ and the isometries do not move the fixed points. By contrast, if $N$ is the complex projective space endowed with the Fubini-Study metric $G$, then $Isom(N, G)$ acts transitively and we can always guarantee the validity of the assumption of Theorem \ref{theotoriche}, so that we recover Hulin's theorem (Remark \ref{condA}).

Notice that if $f \in Aut(N) \setminus Isom(N, G)$ then, in order to guarantee that $f \circ \phi$ is a \K\ embedding one has to replace $G$ by $(f^{-1})^*(G)$, and consequently the torus action, say $\rho$, by $\tilde \rho = f \circ \rho \circ f^{-1}$. Then any new fixed point is of the form $f(p)$, where $p$ is a point fixed by the action $\rho$. This implies that if $\phi(M)$ does not contain any point fixed by $\rho$, then $f(\phi(M))$ does not contain any point fixed by $\tilde \rho$.

\end{remar}

\section{Gromov width of toric varieties}

Let us recall that the Gromov width (introduced in \cite{GROMOV85}) of a $2n$-dimensional symplectic manifold $(M, \omega)$ is defined as 
$$c_G(M, \omega) = \sup \{ \pi r^2 \ | \ (B^{2n}(r), \omega_{can}) \ \text{symplectically embeds into} \ (M, \omega) \}$$
where $\omega_{can} = \frac{i}{2} \sum_{j=1}^n d z_j \wedge d \bar z_j$ is the canonical symplectic form in $\complex^n$.

By Darboux's theorem $c_G(M, \omega)$  is a positive number. Computations
and estimates of the Gromov width for various examples can be found
in \cite{BIRAN99}, \cite{castro}, \cite{GWgrass} and in particular for toric manifolds in \cite{Lu}.

\medskip

In what follows, we are going to make some remarks about the Gromov width of toric manifolds.
More precisely, let $(M, \omega)$ be a toric manifold endowed with an integral toric \K\ form $\omega$. As it is known (\cite{delzant}, \cite{Gu}), the image of the moment map $\mu: M \rightarrow \real^n$ for the isometric action of the real torus $T^n$ on $M$ is a convex {\it Delzant} polytope $\Delta = \{ x \in \real^n \ | \ \langle x, u_i \rangle \geq \lambda_i, \ i=1, \dots, d \} \subseteq \real^n$, i.e. such that the normal vectors $u_i$ to the faces meeting in a given vertex form a $\intero$-basis of $\intero^n$. The vertices of $\Delta$ (which, by the integrality of $\omega$, belong to $\intero^n$) are the images by $\mu$  of the fixed points for the action of $T^n$ on $M$.

\smallskip

As recalled in Section \ref{appendix2} of the Appendix, such a polytope $\Delta$ represents a very ample line bundle on the toric manifold $X_{\Sigma}$ associated to the fan $\Sigma$ which has the $u_i$'s as generators. Then, by the Kodaira embedding $i_{\Delta}$ we can embed $X_{\Sigma}$ into a complex projective space $\complex \projective^{N-1}$ and endow $X_{\Sigma}$ with the pull-back $i_{\Delta}^*( \omega_{FS})$ of the Fubini-Study form $\omega_{FS}= i \log (\sum_{j=1}^N |z_j|^2)$ of $\complex \projective^{N-1}$.  

We have the following, important result.

\begin{teor}\label{equivsymp}  (see, for example, \cite{abreu}, page 3 or \cite{Gu}, Section A2.1) The manifolds $(X_{\Sigma}, i_{\Delta}^*( \omega_{FS}))$ and $(M, \omega)$ are equivariantly symplectomorphic.
\end{teor}

Now, by the following well-known result we can write the Kodaira embedding explicitly.

\begin{prop}\label{prop2prima}
Let $p \in M$ be a fixed point for the torus action and $X_p$, $\phi_p: X_p \rightarrow \complex^n$ be as in Proposition \ref{apertointoriche}.  The restriction to $X_p$ of the Kodaira embedding $i_{\Delta}: M \rightarrow \complex \projective^{N-1}$ writes, in the coordinates given by $\phi_p$,  as

$$i_g|_{X_{p}} \circ\phi_p^{-1}: \complex^n \rightarrow \complex \projective^{N-1},  \ \ \xi \mapsto [\dots, \xi_1^{x_1} \cdots \xi_n^{x_n}, \dots]$$
where $(x_1, \dots, x_n)$ runs over all the points with integral coordinates in the polytope $\Delta'$ of $\real^n$ obtained by $\Delta$ via the transformation in $SL_n(\intero)$ and the translation which send the vertex of $\Delta$ corresponding to $p$ to the origin and the corresponding edge to the edge generated by the vectors $e_1, \dots, e_n$ of the canonical basis of $\real^n$.
\end{prop}

Notice that the existence of the transformation in $SL_n(\intero)$ invoked in the statement  follows from the fact that the normal vectors to the faces meeting in any vertex of the polytope form a $\intero$-basis of $\intero^n$.

We will give a detailed proof of Proposition \ref{prop2prima} in the Appendix (Proposition \ref{prop2})

\smallskip

It follows by Proposition \ref{prop2prima} that the restriction $\omega_{\Delta}$ of the pull-back metric $i_{\Delta}^*( \omega_{FS})$ to the open subset $X_p$ is given in the coordinates $\xi_1, \dots, \xi_n$ by  $i \log (\sum_{j=1}^N |\xi|^{2 J_j}))$, where $\{ J_k \}_{k=1, \dots, N} = \Delta' \cap \intero^n$ and for any $J = (J_1, \dots, J_n) \in \intero^n$ we are denoting $|\xi|^{2J} := |\xi_1|^{2J_1} \cdots |\xi_n|^{2J_n}$.

Then, by combining Theorem \ref{equivsymp} and Proposition \ref{prop2prima}, we conclude that the manifold $(M, \omega)$ has an open dense subset, say $A$, symplectomorphic to $(\complex^n, \omega_{\Delta} := i \log (\sum_{j=1}^N |\xi|^{2 J_j}))$.

We now estimate from above the Gromov width of $(\complex^n, \omega_{\Delta})$. We are going to use the following
\begin{lemma}\label{theoremembedding}
Let $A$ be an open, connected subset of $\complex^n$ such that $A \cap \{ z_j = 0\} \neq \emptyset$, $j =1, \dots,n$, endowed with a K\"ahler form $\omega = \frac{i}{2} \partial \bar \partial \Phi$, where $\Phi(\xi_1, \dots, \xi_n) = \tilde \Phi(|\xi_1|^2, \dots, |\xi_n|^2)$ for some smooth function $\tilde \Phi: \tilde A \rightarrow \real$, $\tilde A = \{(x_1, \dots, x_n) \in \real^n \ | \ x_i = |\xi_i|^2, \ (\xi_1, \dots, \xi_n) \in A \ \}$ (we say that $\omega$ is a rotation invariant form). 
Assume $\frac{\partial \tilde \Phi}{\partial x_k} > 0$ for every $k = 1, \dots, n$. Then the map
\begin{equation}\label{embeddingtheorem}
\Psi: (A, \omega) \rightarrow (\complex^n, \omega_{0}), \ \ (\xi_1, \dots, \xi_n) \mapsto \left( \sqrt{\frac{\partial \tilde \Phi}{\partial x_1}} \xi_1, \dots, \sqrt{\frac{\partial \tilde \Phi}{\partial x_n}} \xi_n \right)  
\end{equation}
is a symplectic embedding (where $\omega_0 = \frac{i}{2} \partial \bar \partial \sum_{k=1}^n |z_k|^2$).
\end{lemma}
For a proof of this lemma, see Theorem 1.1 in \cite{LZ}.
Our result is

\begin{teor}\label{GWCn}
Let $\omega_{\Delta} = i \partial \bar \partial \log (\sum_{j=1}^N |\xi|^{2 J_j})$. Then
\begin{equation}\label{stimaGWCn}
c_G(\complex^n, \omega_{\Delta}) \leq 2 \pi \min_{j=1, \dots, n} \left( \max_{k}\{ (J_k)_j \} \right)
\end{equation}
\end{teor}

\begin{proof}
Let us apply Lemma \ref{theoremembedding} to $A = \complex^n$ endowed with the rotation invariant K\"ahler form $i \partial \bar \partial \log (\sum_{j=1}^N |\xi|^{2 J_j}))$. In the notation of the statement of the lemma, we have then $\tilde \Phi = 2 \log \sum_{k=0}^N x^{J_k}$, where $x = (x_1, \dots, x_n) \in \real^n$ and we are denoting $x^J = x_1^{j_1} \cdots x_n^{j_n}$, for $J = (j_1, \dots, j_n)$. Since for every $k = 0, \dots, N$ we have $J_k = ({(J_k)}_1, \dots, {(J_k)}_n) \in (\intero_{\geq 0})^n$, we have
$$\frac{\partial \tilde \Phi}{\partial x_j} = \frac{\frac{2}{x_j} \sum_{k=0}^N {(J_k)}_j x^{J_k} }{\sum_{k=0}^N x^{J_k}} > 0$$
and then we can embed symplectically $\complex^n$ (endowed with the toric form) into $\complex^n$ (endowed with the standard symplectic form) by $(\xi_1, \dots, \xi_n) \mapsto \left( \sqrt{\frac{\partial \tilde \Phi}{\partial x_1}} \xi_1, \dots, \sqrt{\frac{\partial \tilde \Phi}{\partial x_n}} \xi_n \right)$ so that $(\complex^n, i \partial \bar \partial \log (\sum_{j=1}^N |\xi|^{2 J_j}))$ is symplectomorphic to the domain $D = \Psi(\complex^n) \subseteq \complex^n$ endowed with the canonical symplectic form $\omega_{can}$. 
Now, let $\pi_k: \complex^n \rightarrow \complex$, $\pi_k(z_1, \dots, z_n) = z_k$ denote the projection onto the $k$-th coordinate. Then $D$ is clearly contained in the cylinder $\pi_k(D) \times \C^{n-1} = \{ (z_1, \dots, z_n) \in \complex^n \ | \ \textrm{there exists} \ p \in D \ \textrm{with} \ p_k = z_k \}$ over $\pi_k(D)$, and then in the cylinder 
$$C_R  = \{ (z_1, \dots, z_n) \in \complex^n \ | \ |z_k|^2 < R^2 \},$$ 
where $R$ is the radius of any ball in the $z_k$-plane containing $\pi_k(D)$. By the celebrated Gromov's non-squeezing theorem, which states that the Gromov width of $C_R$ endowed with the canonical symplectic form $\omega_{can}$ is $\pi R^2$, we conclude that the Gromov width of $D$ is less or equal to $\pi R^2$, where $R$ is the radius of any euclidean ball of the $z_k$-plane containing $\pi_k(D)$.

In order to calculate the best value of $R$, notice that
$$\pi_k(D) = \{  \sqrt{\frac{\partial \tilde \Phi}{\partial x_j}} \xi_j \ | \ (\xi_1, \dots, \xi_n) \in \complex^n \} = $$ 
$$= \{  \sqrt{\frac{\partial \tilde \Phi}{\partial x_j} x_j} \ e^{i \theta_j } \ | \ (x_1, \dots, x_n) \in (\real_{\geq 0})^n, \theta_j \in [0, 2 \pi] \}$$
(since $x_j = |\xi_j|^2$ and $\xi_j = \sqrt{x_j} e^{i \theta_j }$) that is the circle in $\real^2$ of radius 
$$\sup \{ \sqrt{\frac{\partial \tilde \Phi}{\partial x_j} x_j} \ | \ (x_1, \dots, x_n) \in (\real_{\geq 0})^n \}.$$
Now, 
\begin{equation}\label{raggio}
\sqrt{\frac{\partial \tilde \Phi}{\partial x_j} x_j}=  \sqrt{ \frac{ 2 \sum_{k=0}^N (J_k)_j x^{J_k} }{\sum_{k=0}^N x^{J_k}}}
\end{equation}
where we are denoting $J_k = ({(J_k)}_1, \dots,  {(J_k)}_n)$. 
Now, fix $j = 1, \dots, n$. On the one hand, we clearly have
$$\sum_{k=0}^N (J_k)_j x^{J_k} \leq \sum_{k=0}^N \max_{k}\{ (J_k)_j \} x^{J_k} = \max_{k}\{ (J_k)_j \} \sum_{k=0}^N  x^{J_k}$$
so that 
$$\sup \sqrt{ \frac{ 2\sum_{k=0}^N (J_k)_j x^{J_k} }{\sum_{k=0}^N x^{J_k}}} \leq \sqrt{2 \max_{k}\{ (J_k)_j \}}.$$
On the other hand, we can show that the sup is actually equal to $\sqrt{2 \max_{k}\{ (J_k)_j \}}$ by setting $x_i = t$ for $i \neq j$ and $x_j = t^s$, for an integer $s$ large enough, and letting $t \rightarrow +\infty$. Indeed, after substituting $x_i = t$ for $i \neq j$ and $x_j = t^s$ we get the one variable function
$$\sqrt{ \frac{ 2 \sum_{k=0}^N (J_k)_j t^{(J_k)_j s + \sum_{i \neq j} (J_k)_i} }{\sum_{k=0}^N t^{(J_k)_j s + \sum_{i \neq j} (J_k)_i}}}$$
and, if we set $f_k(s) = (J_k)_j s + \sum_{i \neq j} (J_k)_i$, it is clear that there is a value of $s$ for which the largest $f_k(s)$ is obtained for the value of $k$ for which $(J_k)_j$ (i.e. the slope of the affine function $f_k(s)$) is maximum.
This concludes the proof.
\end{proof}

As an immediate corollary of Theorem \ref{GWCn}, we get the following

\begin{corol}\label{corollariodivisore}
Let $(M, \omega)$ be a toric manifold endowed with an integral toric form, let $\Delta \subseteq \real^n$ be the image of the moment map for the torus action (which, up to a translation and a transformation in $SL_n(\intero)$, can be assumed to have the origin as vertex and the edge at the origin generated by the canonical basis of $\real^n$) and let $\{ J_k \}_{k=0, \dots, N} = \Delta \cap \intero^n$. Let $p$ be the point fixed by the torus action corresponding to the origin of $\Delta$ and $X_p \simeq \complex^n$ be the open subset given by Proposition \ref{apertointoriche}. 
Then, any ball of radius $r > \sqrt{2 \min_{j=1, \dots, n} \left( \max_{k}\{ (J_k)_j \} \right)}$,symplectically embedded into $(M, \omega)$, must intersect the divisor $M \setminus X_p$.
\end{corol}

Let 
$$\Delta = \{ x \in \real^n \ | \ \langle x, u_k \rangle \geq \lambda_k,  \ k = 1, \dots, d \}.$$
Then Lu proves in Corollary 1.4 of \cite{Lu} that the Gromov width of the corresponding toric manifold is bounded from above by
$$\Lambda(\Delta): = 2 \pi \max \{ - \sum_{i=1}^d \lambda_i a_i  \ | \ a_i \in \intero_{\geq 0}, \  \sum_{i=1}^d a_i u_i = 0,  \ 1 \leq \sum_{i=1}^d a_i \leq n+1 \}$$
in general, and by 
$$\gamma(\Delta): = 2 \pi \inf \{ - \sum_{i=1}^d \lambda_i a_i > 0  \ | \ a_i \in \intero_{\geq 0}, \sum_{i=1}^d a_i u_i = 0 \}$$
if the polytope $\Delta$ is Fano\footnote{It is easy to see that this condition is equivalent to the fact that the \K\ form on the manifold represents the first Chern class of a multiple of the anticanonical bundle.}, that is if there exist $m \in \real^n$ and $r > 0$ such that 
\begin{equation}\label{Fanopolytope}
r(\lambda_i + \langle m, u_i \rangle) = \pm 1, \ i=1, \dots, d, \ \ \ Int(r \cdot (m + \Delta)) \cap \intero^n = \{ 0 \}.
\end{equation}

\begin{Example}\label{esempi}\rm
Take  the polytope

\begin{eqnarray}
\Delta = \{ (x_1, x_2) \in \real^2 \ & | & \ x_1 \geq 0, x_2 \geq 0, x_1 - x_2 \geq -1, x_2 - x_1 \geq -1, \nonumber \\ & & x_1 - 2 x_2 \geq -3, x_2 \leq 3 \} \nonumber
\end{eqnarray}

\noindent which represents a \K\ class $\omega_{\Delta}$ on the Hirzebruch surface $S_2$  blown up at two points, denoted in the following by $\tilde S_2$.

Notice that $\Delta$ is of the kind  $\{ x \in \real^n \ | \ \langle x, u_k \rangle \geq \lambda_k,  \ k = 1, \dots, d \}$, where $u_1 = (1,0), u_2 = (0,1), u_3 = (1,-1), u_4 = (-1,1), u_5 = (1,-2), u_6 = (0, -1)$ and $\lambda_1 = 0, \lambda_2 = 0, \lambda_3 = -1, \lambda_4 = -1, \lambda_5 = -3, \lambda_6 = -3$.
We first show that $\Delta$ does not satisfy the above Fano condition (\ref{Fanopolytope}).
Indeed, these conditions read, for $m=(m_1, m_2)$ and $i = 1,2,3,6$,
$$rm_1 = \pm 1, \ \ rm_2 = \pm 1, \ \ r(-1 + m_1 - m_2) = \pm 1 \ \  r(-3-m_2) = \pm 1 .$$
Combining the second and the last condition we get the four possibilities (the signs have to be taken independently) $-3r-1 = +1, \ \ -3r-1 = -1, \ \ -3r+1 = +1, \ \ -3r+1 = -1$, that is $r = - \frac{2}{3}, r= 0, r = \frac{2}{3}$. Since $r >0$ the only possibility is $r = \frac{2}{3}$, and $m_2 = - \frac{3}{2}$. Replacing this in the third condition, and taking into account the first one, we have 
$$r(-1+m_1-m_2) = \frac{2}{3}(-1 + \frac{3}{2} \pm \frac{3}{2})$$
which is either $\frac{4}{3}$ or $-\frac{2}{3}$, so different from $\pm 1$ for any choice of the signs. This proves the claim. \label{paginafootnote}
Then Lu's estimate by $\gamma(\Delta)$ does not apply\footnote{At page 169 of \cite{Lu}, studying the case of the projective space blown up at one point, the author applies the same estimate valid for Fano polytopes also in the case when the polytope is not Fano: by looking at the proof, it turns out that this is possible because the projective space blown up at one point is Fano and any two \K\ forms on it are deformedly equivalent (see \cite{Salamon}, Example 3.7). This argument cannot be used here because $\tilde S_2$ is not Fano. As told to the second author by D. Salamon in a private communication, it is not known if the same result on the equivalence by deformation holds on the higher blowups of the projective spaces $\complex \projective^n$, with $n > 2$.}. Since $\sum_i a_i u_i = 0$ reads $a_1 = a_4-a_3-a_5, a_2 = a_3 - a_4 +2a_5+a_6$ we have
$$\Lambda(\Delta) = \max \{ 2 \pi (a_3 + a_4 + 3 a_5 + 3 a_6)  \ | \ a_i \in \intero_{\geq 0}, 1 \leq 2 a_5 + 2 a_6 + a_3 + a_4 \leq 3 \}.$$
It is easy to see that $\Lambda(\Delta) = 8 \pi$ (attained for $a_2 = a_4 = a_5 =1$ and $a_1=a_3=a_6=0$). We then get $c_G(\tilde S_2, \omega_{\Delta}) \leq 8 \pi$, while it is easy to see that our estimate (\ref{stimaGWCn}) yields $c_G(\complex^n, \omega_{\Delta}) \leq 6 \pi$.

Then, Corollary \ref{corollariodivisore} in this case states that any ball of radius strictly larger than $\sqrt 6$, simplectically embedded into $(\tilde S_2, \omega_{\Delta})$, must intersect the divisor.

\end{Example}

\begin{Example}\label{esempi2}\rm
Consider the family of polytopes

\begin{eqnarray}
\Delta(m)= \{ (x_1, x_2) \in \real^2 \ & | & \ 0 \leq x_1 \leq 4, 0 \leq x_2 \leq 4, -2 \leq x_1 - x_2 \leq 2, \nonumber \\ &  & 2 x_1 - x_2 \geq -\frac{2m}{m+1} \} \nonumber
\end{eqnarray}

\noindent which, for every natural number $m \geq 1$, represents a \K\ class $\omega_{\Delta(m)}$ on the projective plane blown up at three points and blown up again (at one of the new fixed points by the toric action), which we denote from now on by $M$.

Notice that $\Delta(m)$ is of the kind  $\{ x \in \real^n \ | \ \langle x, u_k \rangle \geq \lambda_k,  \ k = 1, \dots, d \}$, where $u_1 = (1,0), u_2 = (0,1), u_3 = (-1,1), u_4 = (-1,0), u_5 = (0,-1), u_6 = (1, -1), u_7 = (2, -1)$ and $\lambda_1 = 0, \lambda_2 = 0, \lambda_3 = -2, \lambda_4 = -4, \lambda_5 = -4, \lambda_6 = -2, \lambda_7= - \frac{2m}{m+1}$.

One easily sees by a straight calculation as in the previous example that $\Delta(m)$ does not satisfy the above Fano condition (\ref{Fanopolytope}). In fact, it is known (see for example Proposition 2.21 in \cite{Oda}) that $M$ is not Fano, so Lu's estimate by $\gamma(\Delta)$ does not apply (see also the footnote at page \pageref{paginafootnote}). Since $\sum_i a_i u_i = 0$ reads $a_1 = a_3+a_4-a_6-2a_7, a_2 = a_5 + a_6 + a_7 - a_3$ we have
$$\Lambda(\Delta) = \max \{ 2 \pi (2a_3 + 4 a_4 + 4 a_5 + 2 a_6 + \frac{2m}{m+1} a_7)  \}$$
over all the $a_i$'s in $\intero_{\geq 0}$ such that $1 \leq  a_3 + 2 a_4 + 2 a_5 + a_6 \leq 3$.

It is easy to see that $\Lambda(\Delta) = 2 \pi (6 + \frac{2m}{m+1})$ (attained for $a_3 = a_4 = a_7 =1$ and $a_1=a_2=a_5=a_6=0$). We have then $c_G(M, \omega_{\Delta(m)}) \leq 2 \pi (6 + \frac{2m}{m+1})$, while it is easy to see that our estimate (\ref{stimaGWCn}) yields $c_G(\complex^n, \omega_{\Delta(m)}) \leq 8 \pi$ for every $m \geq 1$ (in fact, we need first to multiply $\Delta(m)$ by $m+1$ in order to get an integral polytope for which $\min_{j=1, \dots, n} \left( \max_{k}\{ (J_k)_j \} \right) = 4(m+1)$, and then we rescale by $\frac{1}{m+1}$, and use the fact that $c_G(M, \lambda \omega) = \lambda c_G(M, \omega)$).

Then, Corollary \ref{corollariodivisore} in this case states that any ball of radius strictly larger than $2 \sqrt 2$, simplectically embedded into $(M, \omega_{\Delta(m)})$, must intersect the divisor.

\end{Example}

\begin{remar}
\rm It is worth to notice that, for the complex projective space $\complex \projective^n$ endowed with the Fubini-Study form $\omega_{FS} = i \log(\sum_i |Z_i|^2)$, the Gromov width is known to be equal to $2\pi$ and in fact it is attained by embedding simplectically an open ball of radius $\sqrt 2$ without intersecting the divisor (more precisely, one can see that the image of the symplectic embedding $(\complex^n, \omega_{FS}) \rightarrow (\complex^n, \omega_0)$ given by (\ref{embeddingtheorem}) is exactly a ball of radius $\sqrt 2$).
\end{remar}

\appendix
\section{Toric manifolds}

\subsection{Toric manifolds as compactifications of $\complex^n$}\label{supapp1}

Let us recall the following

\begin{defin}\label{toric}
A {\it toric variety} is a complex variety $M$ containing an open dense subset biholomorphic to $(\complex^*)^n$ and such that the canonical action of $(\complex^*)^n$ on itself by $(\alpha_1, \dots, \alpha_n) (\beta_1, \dots, \beta_n) = (\alpha_1 \beta_1, \dots, \alpha_n \beta_n)$ extends to a holomorphic action on the whole $M$. 
\end{defin}

A toric variety can be described combinatorially by means of {\it fans of cones}. In detail, by the {\it cone $\sigma = \sigma(u_1, \dots, u_m)$ in $\real^n$ generated by the vectors $u_1, \dots, u_m \in \intero^n$} we mean the set 
$$\{ x \in \real^n \ | \ x = \sum_{i=1}^m c_i u_i, \ c_i \geq 0 \}$$
of linear combinations of $u_1, \dots, u_m$ with non-negative coefficients. The $u_i$'s are called the {\it generators} of the cone.  The {\it dimension} of a cone $\sigma = \sigma(u_1, \dots, u_m)$ is the dimension of the linear subspace of $\real^n$ spanned by $\{ u_1, \dots, u_m \}$.

We will always assume that our cones are {\it convex}, i.e. that they do not contain any straight line passing through the origin, and that the generators of a cone are linearly independent.

The {\it faces} of a cone $\sigma = \sigma(u_1, \dots, u_m)$ are defined as the cones generated by the subsets of $\{ u_1, \dots, u_m \}$. By definition, the cone generated by the empty set is the origin $\{ 0 \}$.

\begin{defin}
A {\it fan $\Sigma$ of cones} in $\real^n$ is a set of cones such that

\begin{enumerate}
\item[(i)] for any $\sigma \in \Sigma$ and any face $\tau$ of $\sigma$, we have $\tau \in \Sigma$;
\item[(ii)] any two cones in $\Sigma$ intersect along a common face.
\end{enumerate}
\end{defin}

Let us now recall how one can construct from a fan $\Sigma$ a toric variety.

Let $\{u_1, \dots, u_d \}$, $u_k = (u_{k1}, \dots, u_{kn}) \in {\intero}^n$, be the union of all the generators of the cones in $\Sigma$. For any cone $\sigma = \sigma( \{ u_i \}_{i \in I}) \in \Sigma$, $I \subseteq \{1, \dots, d \}$, let us denote
$${\complex}^d_{\sigma} = \{ (z_1, \dots, z_d) \in {\complex}^d \ | \ z_i = 0 \Leftrightarrow i \in I \} .$$
Notice that if $\sigma = \sigma(\emptyset)$ is the cone consisting of the origin alone, then  ${\complex}^d_{\sigma} = ({\complex}^*)^d$.
Now, let ${\complex}^d_{\Sigma} = \bigcup_{\sigma \in \Sigma} {\complex}^d_{\sigma}$ and $K_{\Sigma}$ be the kernel of the surjective homomorphism
$$\pi: (\complex^*)^d \rightarrow (\complex^*)^n, \ \ \pi(\alpha_1, \dots, \alpha_d) = (\alpha_1^{u_{11}} \cdots \alpha_d^{u_{d1}}, \dots,  \alpha_1^{u_{1n}} \cdots \alpha_d^{u_{dn}}) .$$

\begin{defin}\label{deftoricquotient}
The {\it toric variety $X_{\Sigma}$ associated to $\Sigma$} is defined to be the quotient  
$X_{\Sigma} = \frac{{\complex}^d_{\Sigma}}{K_{\Sigma}}$
of ${\complex}^d_{\Sigma}$ for the action of $K_{\Sigma}$ given by the restriction of the canonical action $(\alpha_1, \dots, \alpha_d) (z_1, \dots, z_d) = (\alpha_1 z_1, \dots, \alpha_d z_d)$ of $(\complex^*)^d$ on $\complex^d$ .
\end{defin}

The importance of this construction consists in the fact that any toric variety $M$ of complex dimension $n$ can be realized as $M = X_{\Sigma}$ for some fan $\Sigma$ in $\real^n$ (see Section 1.4 in \cite{Ful}).

Notice that, by definition of $K_{\Sigma}$, we have $\frac{({\complex}^*)^d}{K_{\Sigma}} \simeq ({\complex}^*)^n$. So we have a natural action of this complex torus on $X_{\Sigma}$ given by
\begin{equation}\label{action}
[(\alpha_1, \dots, \alpha_d)][(z_1, \dots, z_d)] = [(\alpha_1 z_1, \dots, \alpha_d z_d)].
\end{equation}

From now on, and throughout this section, we will assume that $X_{\Sigma}$ is a compact, smooth manifold. From a combinatorial point of view, it is known (\cite{Ful}, Chapter 2) that:

\begin{enumerate}

\item[(i)] $X_{\Sigma}$ is compact if and only if the {\it support $|\Sigma| = \cup_{\sigma \in \Sigma} \sigma$} of $\Sigma$ equals $\real^n$. 

\item[(ii)] $X_{\Sigma}$ is a smooth complex manifold if and only if for each $n$-dimensional cone $\sigma$ in $\Sigma$ its generators form a $\intero$-basis of $\intero^n$.

\end{enumerate}

\label{paginasmooth}
Under these assumptions, we have the following well-known result.

\begin{prop}\label{prop1}
 Let $X_{\Sigma}$ be a compact, smooth toric manifold of complex dimension $n$. Then, for each $p \in X_{\Sigma}$ fixed by the torus action (\ref{action}) there exists an open neighbourhood $X_p$ of $p$, dense in $X_{\Sigma}$, containing the complex torus $\frac{({\complex}^*)^d}{K_{\Sigma}} \simeq ({\complex}^*)^n$ and a biholomorphism $\phi_p: X_p \rightarrow \complex^n$ such that $\phi_p(p)=0$ and that the restriction of the torus action (\ref{action}) to $X_p$ coincides, via $\phi_p$, with the canonical action of $(\complex^*)^n$ on $\complex^n$ by componentwise multiplication.
In particular, any compact, smooth toric manifold of complex dimension $n$ is a compactification of $\complex^n$.
\end{prop}

\begin{proof} Let $\sigma = \sigma(u_{j_1}, \dots, u_{j_n})$ be an $n$-dimensional cone in $\Sigma$, and let $\{ j_{n+1}, \dots, j_{d} \} = \{ 1, \dots, d\} \setminus \{ j_{1}, \dots, j_{n} \}$. Let us consider the open dense subset
$$X_{\sigma} = \frac{\bigcup_{\tau \subseteq \sigma} \complex^d_{\tau}}{K_{\Sigma}} =  \{ [(z_1, \dots, z_d)] \in X_{\Sigma} \ | \ z_{j_{n+1}}, \dots, z_{j_{d}} \neq 0 \}.$$

We are going to define a biholomorphism $\phi_{\sigma}: X_{\sigma} \rightarrow \complex^n$. Recall that, by the assumption of smoothness, $u_{j_1}, \dots, u_{j_n}$ form a $\intero$-basis of $\intero^n$, or equivalently (with a further permutation if necessary) the matrix

$$U = \left( \begin{array}{ccc}
u_{{j_1}1} & \dots & u_{{j_n}1} \\
 \dots & \dots & \dots \\
u_{{j_1}n} & \dots & u_{{j_n}n}
\end{array} \right)$$
belongs to $SL(n, \intero)$. Let $U^{-1} = \left( \begin{array}{ccc}
w_{1 1} & \dots & w_{n1} \\
 \dots & \dots & \dots \\
w_{1 n} & \dots & w_{nn}
\end{array} \right)$ and let
\linebreak $\left( \begin{array}{ccc}
v_{j_{n+1} 1} & \dots & v_{j_{d}1} \\
 \dots & \dots & \dots \\
v_{j_{n+1} n} & \dots & v_{j_{d}n}
\end{array} \right)$
be the matrix in $M_{n \ d-n}(\intero)$ obtained by deleting from 
$$\left( \begin{array}{ccc}
w_{1 1} & \dots & w_{n1} \\
 \dots & \dots & \dots \\
w_{1 n} & \dots & w_{nn}
\end{array} \right)\left( \begin{array}{ccc}
u_{1 1} & \dots & u_{d1} \\
 \dots & \dots & \dots \\
u_{1 n} & \dots & u_{dn}
\end{array} \right)$$
the $j$-th column, for $j = j_1, \dots, j_n$.

We claim that 
\begin{equation}\label{defphi}
\phi_{\sigma}([(z_1, \dots, z_d)]) = (z_{j_{1}} z_{j_{n+1}}^{v_{j_{n+1} 1}} \cdots z_{j_{d}}^{v_{j_{d} 1}}, \dots, z_{j_{n}} z_{j_{n+1}}^{v_{j_{n+1} n}} \cdots z_{j_{d}}^{v_{j_{d} n}})
\end{equation}
defines the required biholomorphism. In order to verify this, notice first that if $(\alpha_1, \dots, \alpha_d) \in K_{\Sigma}$ then, by definition, for every $k=1, \dots, n$, we have
$$1 = (\alpha_1^{u_{11}} \cdots \alpha_d^{u_{d1}})^{w_{1k}} \cdots (\alpha_1^{u_{1n}} \cdots \alpha_d^{u_{dn}})^{w_{nk}} = \alpha_{j_k} \alpha_{j_{n+1}}^{v_{j_{n+1} k}} \cdots \alpha_{j_d}^{v_{j_{d}k}}$$
so that 
\begin{equation}\label{paramK}
\alpha_{j_k} = \alpha_{j_{n+1}}^{-v_{j_{n+1} k}} \cdots \alpha_{j_d}^{-v_{j_{d}k}}, \ \  k=1, \dots, n.
\end{equation}
For any $\alpha_{j_{n+1}}, \dots, \alpha_{j_d} \in \complex^*$, these equations give a parametric representation of  $K_{\Sigma}$, using which it is easy to see that (\ref{defphi}) is well defined. More in detail, if $[(z_1, \dots, z_d)] = [(w_1, \dots, w_d)] $ then there exist $\alpha_{j_{n+1}}, \dots, \alpha_{j_d} \in \complex^*$ such that $w_{j_{n+1}} = \alpha_{j_{n+1}} z_{j_{n+1}}, \dots, w_{j_d} = \alpha_{j_d} z_{j_d}$ and

$$w_{j_1} = \alpha_{j_{n+1}}^{- v_{j_{n+1} 1}} \cdots \alpha_{j_d}^{- v_{j_d 1}} z_1, \ \dots \ , w_{j_n} = \alpha_{j_{n+1}}^{- v_{j_{n+1} n}} \cdots \alpha_{j_d}^{- v_{j_d n}} z_{j_n} $$
from which it is immediate to see that $\phi_{\sigma}[(z_1, \dots, z_d)] = \phi_{\sigma}[(w_1, \dots, w_d)]$.

Moreover, one sees that

\begin{equation}\label{invphi}
\psi_{\sigma}: \complex^n \rightarrow X_{\sigma}, \ \ \ \psi_{\sigma}(\xi_1, \dots, \xi_n) = [(\psi_1, \dots, \psi_d)]
\end{equation}
where
$$\psi_{j_1} = \xi_1, \dots, \psi_{j_n} = \xi_n, \ \ \psi_{j_{n+1}} = \cdots = \psi_{j_d} = 1$$ and 
is the inverse of $\phi_{\sigma}$. Indeed, on the one hand it is clear that $\phi_{\sigma} \circ \psi_{\sigma} = id_{\complex^n}$. On the other hand, for every $[(z_1, \dots, z_d)] \in X_{\sigma}$ we have $(\psi_{\sigma} \circ \phi_{\sigma})([z_1, \dots, z_d]) = [(\psi_1, \dots, \psi_d)]$ where
$$\psi_{j_k} = z_{j_{k}} z_{j_{n+1}}^{v_{j_{n+1} k}} \cdots z_{j_{d}}^{v_{j_{d} k}}, \ \ k = 1, \dots, n$$
and $\psi_{j_{n+1}} = \cdots = \psi_{j_d} = 1$. But $[(\psi_1, \dots, \psi_d)] = [(z_1, \dots, z_d)]$ since $(z_1, \dots, z_d) = (\alpha_1, \dots, \alpha_d)(\psi_1, \dots, \psi_d)$ for the element $(\alpha_1, \dots, \alpha_d) \in K_{\Sigma}$ given by $\alpha_{j_{n+1}} = z_{j_{n+1}}, \dots, \alpha_{j_{d}} = z_{j_{d}}$ (recall that, by definition of $X_{\sigma}$, we have $z_{j_{n+1}}, \dots, z_{j_d} \neq 0$) and 
$$\alpha_{j_k} = z_{j_{n+1}}^{-v_{j_{n+1} k}} \cdots z_{j_{d}}^{-v_{j_{d} k}}, \ \ k = 1, \dots, n.$$
This proves the claim. 
Now, by the very definition of $X_{\sigma}$ it is clear that it contains the complex torus $\frac{({\complex}^*)^d}{K_{\Sigma}}$ and that $X_{\sigma}$ is invariant by the action (\ref{action}). In fact, one has $\phi_{\sigma}\left( \frac{({\complex}^*)^d}{K_{\Sigma}} \right) = (\complex^*)^n$ and, if $\phi_{\sigma}[\alpha_1, \dots, \alpha_d] = (a_1, \dots, a_n)$, $\phi_{\sigma}[z_1, \dots, z_d] = (\xi_1, \dots, \xi_n)$, then $\phi_{\sigma}[\alpha_1 z_1, \dots, \alpha_d z_d] = (a_1 \xi_1, \dots, a_n \xi_n)$, which means that the action of $\frac{({\complex}^*)^d}{K_{\Sigma}}$ on $X_{\sigma}$ corresponds, via $\phi_{\sigma}$, to the canonical action of $(\complex^*)^n$ on $\complex^n$.

As a consequence, since the only fixed point for this canonical action is the origin, we have that the only point of $X_{\sigma}$ fixed by the action of $\frac{({\complex}^*)^d}{K_{\Sigma}}$ is the point $p = [z_1, \dots, z_d]$ having $z_{j_1} = \cdots = z_{j_n} = 0$.

So $X_{\sigma}$ turns out to be a neighbourhood $X_p$ of the fixed point $p$ which satisfies all the requirements of the statement  of the Proposition.

Since the $X_{\sigma}$'s, when $\sigma$ runs over all the $n$-dimensional cones of $\Sigma$, cover $X_{\Sigma}$, we get in this way all the fixed points by the torus action, and this concludes the proof of the Proposition.

\end{proof}

\subsection{Toric bundles and Kodaira embeddings}\label{appendix2}

Let us recall how one constructs combinatorially the line bundles on a toric manifold $X_{\Sigma}$.

\begin{defin}\label{support}
Let $\Sigma$ be a fan of cones in $\real^n$. A {\it $\Sigma$-linear support function} (or simply a support function when the context is clear) is a continuous function $g: \real^n \rightarrow \real$ such that

\begin{enumerate}
\item[(i)] on every $n$-dimensional cone $\sigma \in \Sigma$, g is the restriction of a linear function $g_{\sigma}: \real^n \rightarrow \real$;
\item[(ii)] $g$ has integer values on $\intero^n$.
\end{enumerate}

\end{defin}
A support function is clearly determined by the values it has on the generators of the cones.

One associates to any such function $g$ a line bundle, denoted ${X_{\Sigma}}_g$, on the manifold $X_{\Sigma}$ and defined as  ${X_{\Sigma}}_g = \frac{\complex^d_{\Sigma} \times \complex}{K_{\Sigma}}$
where $\complex^d_{\Sigma}$, $K_{\Sigma}$ are as in Definition \ref{deftoricquotient} and the quotient comes from the action of $K_{\Sigma}$ on $\complex^d_{\Sigma} \times \complex$ given by 

$$(\alpha_1, \dots, \alpha_d) \cdot (z_1, \dots, z_d, z_{d+1}) = (\alpha_1 z_1, \dots, \alpha_d z_d, \alpha_1^{-g(u_1)} \cdots \alpha_d^{-g(u_d)} z_{d+1}).$$

The projection $p:{X_{\Sigma}}_g \rightarrow X_{\Sigma}$ is just given by $p([z_1, \dots, z_{d+1}]) = [z_1, \dots, z_d]$, which is clearly well-defined by the very definition of the equivalence relations involved.

It is known that ${X_{\Sigma}}_g$ is very ample if and only if $g$ is {\it strictly convex}, i.e. it fulfills the following requirements:

\begin{enumerate}
\item[(a)] for every $v_1, v_2 \in \real^n$, $t \in [0,1]$, one has $g(tv_1 + (1-t)v_2) \geq t g(v_1) + (1-t) g(v_2)$ (i.e. $-g$ is convex);
\item[(b)] distinct $n$-dimensional cones $\sigma$ give distinct functions $g_{\sigma}$. 
\end{enumerate}

A nice representation of the very ample line bundle $p:{X_{\Sigma}}_g \rightarrow X_{\Sigma}$, encoding combinatorially both the structure of $X_{\Sigma}$ and the function $g$, is given by the convex polytope
\begin{equation}\label{polytope}
\Delta_g = \{ x \in \real^n \ | \ \langle x, u_i \rangle \geq g(u_i), \ \ i=1, \dots, d \}
\end{equation}
where $u_1, \dots, u_d$ are the generators of $\Sigma$.

Every $k$-dimensional face of $\Delta_g$ is given by the intersection of $n-k$ hyperplanes $\langle x, u_i \rangle = g(u_i)$, for $i \in I \subseteq \{1, \dots, d \}$ such that $\{ u_i \}_{i \in I}$ generates an $(n-k)$-dimensional cone of $\Sigma$. In particular, the vertices of $\Delta_g$ correspond to the $n$-dimensional cones of $\Sigma$ and then (see the proof of Proposition \ref{prop1}) to the fixed points of the torus action.

Conversely, every convex polytope $\Delta = \{ x \in \real^n \ | \ \langle x, u_i \rangle \geq \lambda_i, \ \ i=1, \dots, d \}$ with the property that the normal vectors $u_i$ to the faces meeting in a given vertex  form a $\intero$-basis of $\intero^n$ determine a toric manifold together with a very ample line bundle.

\smallskip

We are now ready to prove the following

\begin{prop}\label{prop2}
Let $p \in X_{\Sigma}$ be a fixed point for the torus action and $X_p$, $\phi_p: X_p \rightarrow \complex^n$ be as in Proposition \ref{prop1}. 

The restriction to $X_p$ of the Kodaira embedding $i_g: X_{\Sigma} \rightarrow \complex \projective^{N-1}$ associated to ${X_{\Sigma}}_g$ writes, in the coordinates given by $\phi_p$,  as

$$i_g|_{X_{p}} \circ\phi_p^{-1}: \complex^n \rightarrow \complex \projective^{N-1},  \ \ \xi \mapsto [\dots, \xi_1^{x_1} \cdots \xi_n^{x_n}, \dots]$$
where $(x_1, \dots, x_n)$ runs over all the points with integral coordinates in $\Delta$, being $\Delta$ the polytope in $\real^n$ obtained by $\Delta_g$ via the transformation in $SL_n(\intero)$ and the translation which send the vertex of $\Delta_g$ corresponding to $p$ to the origin and the corresponding edge to the edge generated by the vectors $e_1, \dots, e_n$ of the canonical basis of $\real^n$.
\end{prop}

\begin{proof} For the sake of simplicity and without loss of generality, we can assume that the fixed point $p$ corresponds (in the sense of the proof of Proposition \ref{prop1}) to the $n$-dimensional cone of $\Sigma$ generated by $u_1, \dots, u_n$, so that $X_p =  \{ [(z_1, \dots, z_d)] \in X_{\Sigma} \ | \ z_{n+1}, \dots, z_d \neq 0 \}$. Given the line bundle $p: {X_{\Sigma}}_g \rightarrow X_{\Sigma}$, we clearly have
$$p^{-1}(X_p) = \{ [(z_1, \dots, z_d, z_{d+1})] \in {X_{\Sigma}}_g \ | \ z_{n+1}, \dots, z_d \neq 0 \}.$$

An explicit trivialization $f: p^{-1}(X_p) \rightarrow X_p \times \complex$ of ${X_{\Sigma}}_g$ on $X_p$ is given by
$$f([(z_1, \dots, z_d, z_{d+1})]) = ([z_1, \dots, z_d], z_{d+1} z_{n+1}^{c_{n+1}} \cdots z_{d}^{c_{d}})$$
where, for every $j = n+1, \dots, d$, 
$$c_j = g(u_{j}) - \sum_{k=1}^n v_{j k} g(u_k)$$
and the $v_{jk}$'s are defined in the proof of Proposition \ref{prop1}.

Indeed,  $f$ is well defined because $z_{n+1}, \dots, z_d \neq 0$ and because, if $[(z_1, \dots, z_d, z_{d+1})] = [(w_1, \dots, w_d, w_{d+1})]$ then, for some $(\alpha_1, \dots, \alpha_d) \in K_{\Sigma}$,
$$w_{d+1} w_{n+1}^{c_{n+1}} \cdots w_{d}^{c_{d}} = (z_{d+1} \alpha_1^{-g(u_1)} \cdots \alpha_d^{-g(u_d)}) z_{n+1}^{c_{n+1}} \cdots z_{d}^{c_{d}}  \alpha_{n+1}^{c_{n+1}} \cdots \alpha_{d}^{c_{d}} =$$
$$= z_{d+1} z_{n+1}^{c_{n+1}} \cdots z_{d}^{c_{d}} \cdot \alpha_1^{-g(u_1)} \cdots \alpha_n^{-g(u_n)} \alpha_{n+1}^{- \sum_{k=1}^n v_{n+1 k} g(u_k)} \cdots \alpha_d^{- \sum_{k=1}^n v_{d k} g(u_k)} $$
$$= z_{d+1} z_{n+1}^{c_{n+1}} \cdots z_{d}^{c_{d}}$$
by (\ref{paramK}) in the proof of Proposition \ref{prop1}.

The inverse of $f$ is clearly given by $f^{-1}: X_p \times \complex \rightarrow p^{-1}(X_p)$, 
$$f^{-1}([z_1, \dots, z_d], z) = [(z_1, \dots, z_d, z z_{n+1}^{-c_{n+1}} \cdots z_{d}^{-c_{d}})],$$ 
which is well defined by the same arguments as above.

A section of $p: {X_{\Sigma}}_g \rightarrow X_{\Sigma}$  is determined by a function $F = F(z_1, \dots, z_d)$ which satisfies 
$$F(\alpha_1 z_1, \dots, \alpha_d z_d) = \alpha_1^{-g(u_1)} \cdots \alpha_d^{-g(u_d)} F(z_1, \dots, z_d).$$
for every $(\alpha_1, \dots, \alpha_d) \in K_{\Sigma}$. Indeed, this is exactly the condition which assures that $s: X_{\Sigma} \rightarrow {X_{\Sigma}}_g$, $s([z_1, \dots, z_d]) = [(z_1, \dots, z_d, F(z_1, \dots, z_d))]$ is well-defined.

By a straight calculation and by (\ref{paramK}), a basis for the space of global sections is given by the polynomials $F(z_1, \dots, z_d) = z_1^{x_1} \cdots z_d^{x_d}$, $x_i \geq 0$, which satisfy
\begin{equation}\label{conditionsection}
x_j + g(u_j) = v_{j1} (x_1 + g(u_1)) + \cdots + v_{jn} (x_n + g(u_n)), \ \ \ j = n+1, \dots, d
\end{equation}
where the $v_{jk}$'s are defined in the proof of Proposition \ref{prop1}. We will refer to this basis as the {\it monomial basis}.
Let $\{F_0, \dots, F_N \}$ be the monomial basis. Then, by the celebrated theorem of Kodaira, the map
\begin{equation}\label{ig}
X_{\Sigma} \stackrel{i_g}{\longrightarrow} \complex P^{N}, \ \ \ [(z_1, \dots, z_d)] \mapsto [F_0(z_1, \dots, z_d), \dots, F_N(z_1, \dots, z_d)].
\end{equation}
yields an embedding of $X_{\Sigma}$ in the complex projective space.
Restricting $i_g$ to $X_p$ and composing with $\phi_p^{-1}: \complex^n \rightarrow X_p$ we get

\begin{equation}\label{embCn}
(\xi_1, \dots, \xi_n) \mapsto [F_0(\xi_1, \dots, \xi_n, 1 \dots, 1), \dots, F_N(\xi_1, \dots, \xi_n, 1 \dots, 1)].
\end{equation}

Now, since the $x_i$'s are all non-negative integers, conditions (\ref{conditionsection}) are equivalent to
\begin{equation}\label{newconditionsection}
\langle x + g_u, v_j \rangle \geq g(u_j), \ \ j=1, \dots, d,
\end{equation}
where $x = (x_1, \dots, x_n)$, $g_u = (g(u_1), \dots, g(u_n))$ and for $j=1, \dots, n$ we are setting $v_j = e_j$ (the canonical basis of $\real^n$). 
Since $e_1, \dots, e_n, v_{n+1}, \dots, v_d$ are the images of $u_1, \dots, u_n, u_{n+1}, \dots, u_d$ via the map \linebreak $A = \left( \begin{array}{ccc}
u_{{1}1} & \dots & u_{{n}1} \\
 \dots & \dots & \dots \\
u_{{1}n} & \dots & u_{{n}n}
\end{array} \right)^{-1} \in SL_n(\intero)$, one easily sees that (\ref{newconditionsection}) are the defining equations of the polytope $\Delta = {}^T A^{-1} (\Delta_g) - g_u$, obtained from $\Delta_g = \{x \in \real^n \ | \ \langle x, u_i \rangle \geq g(u_i) \}$  by the map in $SL_n(\intero)$ and the translation which send the edge given by the faces having $u_1, \dots, u_n$ as normals (i.e. the edge at the vertex corresponding to $p$) to the edge at the origin having the vectors of the canonical basis as edge.
Then the embedding (\ref{embCn}) turns out to be the map
$$\complex^n \rightarrow \complex P^{N}, \ \ \  (\xi_1, \dots, \xi_n) \mapsto [\dots, \xi_1^{x_1} \cdots \xi_n^{x_n}, \dots]$$
where $(x_1, \dots, x_n)$ runs over all the points with integral coordinates in $\Delta$, as required.
\end{proof}

\begin{remar}\label{puntipolitopo}\rm 
Notice that the transformed polytope $\Delta$ represents, up to isomorphism, the same line bundle and the same toric manifold as $\Delta_g$, because we are always free to apply to the fan $\Sigma$ a transformation in $SL_n(\intero)$ (see, for example, Proposition VII.1.16 in \cite{Au}) and we can always add to $g$ an integral linear function $\intero^n \rightarrow \intero$, which correspond to a translation of the polytope (this comes from the fact that two bundles ${X_{\Sigma}}_g$, ${X_{\Sigma}}_{g'}$ associated to $\Sigma$-linear support functions $g$, $g'$ on $\Sigma$ are isomorphic if and only if $g-g': \real^n \rightarrow \real$ is a linear function).

In fact, it is well-known that if a toric manifold endowed with a very ample line bundle is represented by a polytope $\Delta$, then the monomials $a_1^{y_1} \cdots a_n^{y_n}$, for $(y_1, \dots, y_n) \in \Delta \cap \intero^n$ and $(a_1, \dots, a_n) \in (\complex^*)^n$, give the restriction of a Kodaira embedding (associated to the given line bundle) to the complex torus $(\complex^*)^n$ contained in $X_{\Sigma}$. What we have seen here in detail is exactly that this embedding can be extended to $\complex^n$ if the polytope has the origin as vertex and the edge at the origin is generated by the canonical basis of $\real^n$, and coincides with (\ref{embCn}) in this case.

\end{remar}

\end{document}